\documentclass[11pt,letterpaper]{article}
\usepackage{amsfonts, amsmath, amssymb, amscd, amsthm, color, graphicx, mathrsfs, wasysym, setspace, mdwlist}
\usepackage{hyperref}
\newcommand{\G}{\mathcal G}

\newcommand{\AH}{\mathcal{AH}}

\newcommand{\e}{\varepsilon}

\renewcommand{\d}{{\rm d}}

\newtheorem{thm}{Theorem}
\newtheorem{cor}[thm]{Corollary}
\newtheorem{lem}[thm]{Lemma}

\newtheorem{q}[thm]{Question}

\theoremstyle{definition}
\newtheorem{defn}[thm]{Definition}
\theoremstyle{remark}

\newfont{\eufm}{eufm10}

\begin{document}

\title{Erratum to the paper ``Acylindrical hyperbolicity of groups acting on trees"}

\author{A. Minasyan, D. Osin}

\date{}

\maketitle

In Lemma 3.9 of \cite{MO}, we stated that acylindrical hyperbolicity of a group is invariant under commensurability up to finite kernels (for definitions and background material we refer to \cite{MO} and \cite{Osi16b}). This lemma was used to prove some of the main results in \cite{MO} and the subsequent paper \cite{Osi16b}. Unfortunately, its proof contains a gap. The goal of this erratum is to point out the gap and correct the statements and proofs of results of \cite{MO} and \cite{Osi16b} affected by it. The arXiv versions of the papers \cite{MO} and \cite{Osi16b} will be updated accordingly.

\paragraph{1.1. The gap.}
The arguments given in \cite{MO} correctly prove some parts of Lemma 3.9, which are summarized below.

\begin{lem}\label{39}
Let $H$ be an acylindrically hyperbolic group. Suppose that $G$ is a finite index subgroup of $H$, or a quotient of $H$ modulo a finite normal subgroup, or an extension of $H$ with finite kernel. Then $G$ is also acylindrically hyperbolic.
\end{lem}

However, we do not know the answer to the following.

\begin{q}\label{q}
Suppose that a group $G$ contains a finite index subgroup which is acylindrically hyperbolic. Does it follow that $G$ is acylindrically hyperbolic itself?
\end{q}

The attempt to give the affirmative answer in the first paragraph of the proof of Lemma~3.9 of \cite{MO} is unsuccessful. Indeed, it claims that if $X$ is a generating set of $H$ and $Y$ is a finite set  of representatives of cosets of $H$ in $G$, then natural map between the Cayley graphs $\Gamma (H, X) \to \Gamma (G, X \cup Y )$ is a quasi-isometry. In general, this is false if $X$ is infinite. The simplest counterexample is $$G=\langle a, b, t\mid t^2=1,\; t^{-1}at=b\rangle ,\;\;\; H=\langle a, b\rangle,\;\;\; X=\{b\}\cup \langle a\rangle .$$ In this case one can take $Y=\{ 1, t\}$. Then for all $n\in \mathbb Z$, we have $|b^n|_X=n$ while $|b^n|_{X\cup Y}= |t^{-1}a^nt|_{X\cup Y}\le 3$

\paragraph{1.2. Necessary corrections in \cite{MO}.} Lemma 3.9 was used in the proof of Theorem 5.6 to deal with non-orientable manifolds. In turn, Theorem 5.6 was used to derive a number of corollaries, including Theorem 2.8 and Corollaries 2.9-2.11 and 7.1-7.3. Note, however, that Lemma 3.9 is not required for dealing with orientable manifolds. Thus we have the following corrected statement of \cite[Theorem 5.6]{MO}:

\begin{thm}\label{thm:new_5.6} Let $M$ be a compact orientable $3$-manifold and let $G$ be a subgroup of $\pi_1(M)$.
Then exactly one of the following three conditions holds.
\begin{itemize}
  \item[{\rm (I)}] $G$ is acylindrically hyperbolic with trivial finite radical.
  \item[{\rm (II)}] $G$ contains an infinite cyclic normal subgroup $Z$ and $G/Z$ is acylindrically
hyperbolic. In fact, $G/Z$ is virtually a subgroup of a surface group in this case.
  \item[{\rm (III)}] $G$ is virtually polycyclic.
\end{itemize}
\end{thm}

\emph{Similarly, the proofs of Theorem 2.8 and Corollaries 2.9-2.11, 7.1-7.3 from \cite{MO} remain valid if all the manifolds are assumed to be orientable. 
In particular, we re-define
$\mathcal M^3$ as the class of groups consisting of subgroups of fundamental groups of compact {orientable} $3$-manifolds.}

Another result of \cite{MO} whose proof involves Lemma 3.9 is Corollary 2.16. Its statement should be modified as follows.

\begin{cor} The conditions $G\in \mathcal C_{reg}$,  $G\in \mathcal D_{reg}$, and $G$ is acylindrically hyperbolic are equivalent for any group $G$ from the following classes.
\begin{enumerate}
\item[(a)] Subgroups of fundamental groups of compact orientable $3$-manifolds.
\item[(b)] Subgroups of graph products of amenable groups. In particular, this class includes subgroups of right angled Artin groups.
\end{enumerate}
\end{cor}

The proof of Corollary 2.16 provided in \cite{MO} becomes correct after the following modification: instead of proving the result for a  group $K$ which is commensurable up to a finite kernel to a group $H$ of type (a) or (b), we prove the result for $H$ itself. Then the reference to Lemma 3.9 in the third paragraph of the proof can be removed.

\paragraph{1.3. Finite extensions of $\AH$-accessible groups.}
In fact, it is also possible to prove the above mentioned results of \cite{MO} for fundamental groups of closed (but possibly non-orientable) manifolds using tools from \cite{ABO}. Let us briefly recall the necessary notation and terminology introduced in \cite{ABO}. 

Given a group $H$ and generating sets $X$, $Y$ of $H$, we write $X\preceq Y$ if the identity map on $H$ induces a Lipschitz map between the metric spaces $(H, \d_Y)\to (H, \d_X)$, where $\d_X$ and $\d_Y$ denote the corresponding word metrics. Then $\preceq $ is a preorder which induces an equivalence relation and an ordering on the equivalence classes in the standard way. The obtained poset of equivalence classes of generating sets of $H$ is denoted by $\G(H)$. By $\AH(H)$ we denote the subset of $\G (H)$ consisting of equivalence classes $[X]$ such that the Cayley graph $\Gamma (H,X)$ is hyperbolic and the natural action of $H$ on $\Gamma (H,X)$ is acylindrical. The set $\AH(H)$ endowed with the order inherited from $\G (H)$ is called the \emph{poset of acylindrically hyperbolic structures} on $H$.

\begin{defn}
A group $H$ is called $\AH$\emph{-accessible} if the poset $\AH(H)$ contains the largest element.
\end{defn}

It is shown in \cite[Theorem 2.18]{ABO} that there exist finitely presented groups which are not $\AH$-accessible. On the other hand, many groups of geometric origin are $\AH$-accessible; these include fundamental groups of closed orientable $3$-manifolds, see \cite[Theorem 2.19]{ABO}.

We first show that the answer to Question \ref{q} is affirmative in the special case when the finite index subgroup of $G$ is normal and $\AH$-accessible.

\begin{lem}\label{AHacc}
Suppose that a group $G$ contains a normal acylindrically hyperbolic subgroup $H$ of finite index which is $\AH$-accessible. Then $G$ is acylindrically hyperbolic.
\end{lem}

\begin{proof}
We will show that the original strategy suggested in the proof of Lemma 3.9 in \cite{MO} works in this case. Given a generating set $Z$ of a group, we denote by $|\cdot |_Z$ the corresponding length function. Let us fix a (finite) set of representatives of $H$-cosets in $G$, such that the representative of $H$ is $1$, and denote it by $Y$.

It is easy to show (see Lemma 5.23 in \cite{ABO}) that the formula $\alpha ([X])=[\alpha (X)]$ for all $\alpha \in Aut (H)$ and $[X]\in \AH(H)$ gives a well-defined order preserving action of $Aut(H)$ on $\AH(H)$. From now on, let $[X]\in \AH (H)$ denote the largest element. Since the largest element is unique, $[X]$ is fixed by every automorphism of $H$. In particular, there exists a constant $A$ such that
\begin{equation}\label{A}
|y^{-1}x^{-1}y|_X=|y^{-1}xy|_X\le A
\end{equation}
for all $x\in X$ and all $y\in Y$.

Let us first prove that the natural inclusion of Cayley graphs $\Gamma (H,X)\to \Gamma (G, X\cup Y)$ is an $H$-equivariant quasi-isometry. This inclusion is quasi-surjective since $H$ has finite index in $G$, hence it remains to check that there exists a constant $C>0$ such that
\begin{equation}\label{C}
|h|_X \le C|h|_{X\cup Y}
\end{equation}
for all $h\in H$.

Let $a_1\ldots a_n$ be a shortest word in the alphabet $X^{\pm 1}\cup Y^{\pm 1}$ representing an element $h\in H$ in $G$. For every $i=0, \ldots , n$, there exists $y_i\in Y$ such that $w_iy_i\in H$, where $w_0=1$ and $w_i=a_1\ldots a_i$; of course, we have $y_0=y_n=1$ as $h\in H$. Then we have
$$
h=(y_0^{-1}a_1y_1) (y_1^{-1}a_2y_2) \cdots (y_{n-1}^{-1} a_ny_n).
$$
Let $h_i=y_{i-1}^{-1}a_iy_i$. It is easy to see by induction on $i$ that $h_i\in H$ for all $i$. Thus to prove (\ref{C}) it suffices to show that
\begin{equation}\label{hC}
|h_i|_X \le C
\end{equation}
for all $i$. We consider two cases. First, assume that $a_i\in X^{\pm 1}$. Then $a_i\in H$ and since $H\lhd G$ and $y_{i-1}^{-1}a_iy_i\in H$, we must have $y_{i-1}=y_i$, so the inequality (\ref{hC}) follows from (\ref{A}) for any $C \ge A$. Next, assume that $a_i\in Y^{\pm 1}$. Since $Y$ is finite, and $|h_i|_Y\le 3$ in this case, we can guarantee (\ref{hC}) for the finite set of such elements $h_i$ by taking $C$ large enough.

Thus we have an $H$-equivariant quasi-isometry $\Gamma (H,X)\to \Gamma (G, X\cup Y)$. Since the Cayley graph $\Gamma (H, X)$ is hyperbolic, the same follows for 
$\Gamma (G, X\cup Y)$. It remains to show that the natural action of $G$ on $\Gamma (G, X\cup Y)$ is acylindrical.

Let $\d_{X\cup Y}(\cdot,\cdot)$ and $\d_{X}(\cdot,\cdot)$ denote the standard edge-path metrics on $\Gamma (G, X\cup Y)$ and $\Gamma (H, X)$ respectively.
For any $\varepsilon >0$ we need to find constants $R$ and $N$ such that given any element $z \in G$, with $\d_{X\cup  Y}(1,z) \ge R$, 
the size of the set $$\Omega_{\e,R}(z)=\{g \in G \mid \d_{X\cup Y}(1,g) \le \e \mbox{ and  } \d_{X\cup Y}(z,gz) \le \e\}$$
is bounded by $N$. Since $|Y|<\infty$ it is enough to show that there is $N' \in \mathbb{N}$ such that $|\Omega_{\e,R}(z) \cap yH|\le N'$ for each $y \in Y$.

Consider any $y \in Y$ with $\Omega_{\e,R}(z) \cap yH\neq \emptyset$, and fix some $h_0 \in H$ such that $yh_0 \in \Omega_{\e,R}(z)$.
By the definition of $Y$  there is $w \in H$ such that $\d_{X\cup Y}(z,w) \le 1$, hence 
$\d_{X\cup Y}(gz,gw)\le 1$ for all $g \in G$. Therefore, given any $yh \in \Omega_{\e,R}(z)$, we have
\begin{multline*}
\d_{X\cup Y}(yh_0w,yhw)\le \d_{X\cup Y}(yh_0w,yh_0z)+\d_{X\cup Y}(yh_0z,z)+\d_{X\cup Y}(z,yhz)+\d_{X\cup Y}(yhz,yhw)\\ \le 1+\e+\e+1=2\e+2.
\end{multline*}

Consequently, $\d_{X\cup Y}(w,h_0^{-1}hw) \le 2\e+2$ whenever $yh \in \Omega_{\e,R}(z)$. We also see that in this case
\[\d_{X\cup Y}(1,h_0^{-1}h)=\d_{X\cup Y}(yh_0,yh)\le \d_{X\cup Y}(yh_0,1)+\d_{X\cup Y}(1,yh)\le 2\e.\] 
Now, the inequality \eqref{C} implies that in $\Gamma(H,X)$ we have
\begin{equation}\label{eq:d_X-1}
\d_X(w, h_0^{-1}hw) \le 2C\e+2C  \mbox{ and } \d_X(1, h_0^{-1}h) \le 2C\e \le 2C\e+2C.  
\end{equation}
On the other hand,
\begin{equation}\label{eq:d_X-2}
\d_X(1,w) \ge \d_{X\cup Y}(1,w) \ge \d_{X\cup Y}(1,z)-\d_{X\cup Y}(z,w)\ge R-1.
\end{equation}
Since the action of $H$ on $\Gamma(H,X)$ is acylindrical, there exist $R', N' \in \mathbb{N}$, depending only on $\e'=2C\e+2C $, such that whenever $\d_X(1,w) \ge R'$ the set of elements $h_0^{-1}h \in H$ satisfying \eqref{eq:d_X-1} has size at most $N'$. Since $C$ is a fixed constant, in view of \eqref{eq:d_X-2} 
we can choose $R =R'+1$ to ensure that
$|\Omega_{\e,R}(z) \cap yH|\le N'$, as required.

Therefore the action of $G$ on $\Gamma (G, X\cup Y)$ is acylindrical, thus
$[X\cup Y]\in \AH (G)$. Since $H$ is acylindrically hyperbolic and $[X]\in \AH(H)$ is maximal, the action of $H$ on $\Gamma (H,X)$ is non-elementary by 
\cite[Theorem 1.1]{Osi16a}. Hence so is the action of $G$ on $\Gamma (G, X\cup Y)$. In particular, the group $G$ is acylindrically hyperbolic.
\end{proof}

\paragraph{1.4. Non-orientable closed $3$-manifold groups.}
Lemma \ref{AHacc}, together with the fact that fundamental groups of closed orientable $3$-manifolds are $\AH$-accessible from \cite[Theorem~2.19]{ABO}, imply the following: to show that the fundamental group of a non-orientable closed manifold is acylindrically hyperbolic it suffices to show this for the fundamental group of its orientable double cover. Replacing references to Lemma 3.9 with this observation in the proof of \cite[Theorem 5.6]{MO}, we obtain the following analogue of Theorem \ref{thm:new_5.6} 
for non-orientable manifolds.

\begin{thm}\label{thm:no_5.6} Let  $G$ be the fundamental group of a closed non-orientable $3$-manifold.
Then exactly one of the following three conditions holds.
\begin{itemize}
  \item[{\rm (I)}] $G$ is acylindrically hyperbolic with trivial finite radical.
  \item[{\rm (II)}] $G$ contains an infinite cyclic normal subgroup $Z$ and $G/Z$ is acylindrically
hyperbolic. In fact, $G/Z$ is virtually a subgroup of a surface group in this case.
  \item [{\rm (III)}] $G$ is virtually polycyclic.
\end{itemize}
\end{thm}

Note that, unlike Theorem \ref{thm:new_5.6}, Theorem \ref{thm:no_5.6} does not cover \emph{subgroups} of fundamental groups of non-orientable $3$-manifolds and also assumes that 
the $3$-manifolds are \emph{closed}. It follows that

\emph{Theorem 2.8, Corollaries 2.9-2.11 and 7.1-7.3, and Corollary 2.16 (a) from \cite{MO} hold for fundamental groups of closed non-orientable manifolds (with $\mathcal{M}^3$
replaced by the class of such fundamental groups).}

\paragraph{1.5. Necessary corrections in \cite{Osi16b}.}
In Theorem 1.6 and Corollary 1.4, the assumption that $G$ virtually surjects onto $\mathbb Z$ should be replaced with the assumption that $G$ itself surjects onto $\mathbb Z$. The conclusion of Corollary 1.2 should be changed to ``then $G$ is virtually acylindrically hyperbolic". Other results remain true as stated.

In addition, the proof of Theorem 1.6 should be modified (and simplified) as follows. Let $G$ be a finitely presented group with $\beta_1^{(2)} (G)>0$ which admits a surjective homomorphism $\e\colon G\to \mathbb Z$. It is shown in the first paragraph of the proof of Lemma 3.1 in \cite{Osi16b} that $G$ splits as an HNN-extension of a finitely generated group $A$ with finitely generated associated subgroups $C$ and $D$ and the stable letter $t$ such that $\e(t)$ is a generator of $\mathbb Z$.

We first observe that $C\ne A\ne D$. Indeed, assume $A=C$ or $A=D$. Let $K_n$ denote the preimage of $n\mathbb Z$ under $\e$. Then $\beta^{(2)}_1(K_n)=n\beta_1^{(2)}(G)\to \infty $ as $n\to \infty$. On the other hand, it is straightforward to check that $K_n$ is generated by $A$ and $t^n$, which yields a uniform upper bound on the number of generators of $G$. Since the first $\ell^2$-Betti number of any group is bounded above by the number of generators minus $1$, we get a contradiction.

Thus $C\ne A\ne D$. By a theorem of Peterson and Thom \cite[Theorem 5.12]{PT}, a countable group with positive first $\ell^2$-Betti number cannot contain finitely generated $s$-normal subgroups of infinite index. In particular, $C$ is not $s$-normal in $G$, i.e., there exists $g\in G$ such that $|g^{-1}Cg\cap C|<\infty$. Now applying Corollary 2.2 from \cite{MO} we conclude that $G$ is acylindrically hyperbolic.


\begin{thebibliography}{99}

\bibitem{ABO} C. Abbott, S. Balasubramanya, D. Osin, Hyperbolic structures on groups. \texttt{arXiv:1710.05197}

\bibitem{MO} A. Minasyan, D. Osin, Acylindrical hyperbolicity of groups acting on trees, \emph{Math. Ann.} \textbf{362} (2015), no. 3-4, 1055--1105. \texttt{arXiv:1310.6289}

\bibitem{Osi16a} D. Osin, Acylindrically hyperbolic groups, \emph{Trans. Amer. Math. Soc.} \textbf{368} (2016), no. 2, 851--888.

\bibitem{Osi16b} D. Osin, On acylindrical hyperbolicity of groups with positive first $\ell^2$-Betti number,  \emph{Bull. Lond. Math. Soc.} \textbf{47} (2015), no. 5, 725--730. 	\texttt{arXiv:1501.03066}

\bibitem{PT}
J. Peterson, A. Thom, Group cocycles and the ring of affiliated operators, {\it Invent. Math.}, {\bf 185}, (2011), no. 3, 561--592.

\end{thebibliography}
\end{document}